\documentclass{amsart}
\usepackage{amssymb}
\usepackage{amsthm}
\newcommand{\pn}{\par\noindent}
\newcommand{\pmn}{\par\medskip\noindent}
\newcommand{\pbn}{\par\bigskip\noindent}
\newtheorem{theor}{Theorem}[section]
\newtheorem{prop}{Proposition}[section]
\theoremstyle{definition} 
\newtheorem{defin}{Definition}[section]
\newtheorem{ex}{Example}[section] \theoremstyle{remark}
\newtheorem{rem}{Remark}[section]
\begin{document}
\title{On arithmetic of plane trees}
\author{Yury Kochetkov}
\begin{abstract} In \cite{Zapponi} L. Zapponi studied the
arithmetic of plane bipartite trees with prime number of edges. He
obtained a lower bound on the degree of a tree's definition field.
Here we obtain a similar lower bound in the following case. There
exists a prime $p$ such, that: a) the number of edges is divisible
by $p$, but not by $p^2$; b) for any proper subset of white (or
black) vertices the sum of their degrees is not divisible by this
$p$.
\end{abstract} \email{yukochetkov@hse.ru, yuyukochetkov@gmail.com}
\maketitle

\section{Introduction}
\pn Let $T$ be a plane (i.e. embedded into plane) bipartite tree
with $N$ edges, $n$ white vertices $v_1,\ldots,v_n$ and $m$ black
ones $u_1,\ldots,u_m$: $n+m=N+1$. Let $k_1,\ldots,k_n$ be degrees
of vertices $v_1,\ldots,v_n$, respectively, and $l_1,\ldots,l_m$
be degrees of $u_1,\ldots,u_m$. Thus, $P=\langle
k_1,\ldots,k_n\,|\,l_1,\ldots, l_m\,\rangle$ is the passport of
the tree $T$. \pmn We will assume that $N=p\,r$, where $p$ is a
prime and $r$ is coprime with $p$. \pmn If there exists a proper
subset
$$\{k_{i_1},\ldots,k_{i_t}\}\subset\{k_1,\ldots,k_n\}$$ such, that
the sum $k_{i_1}+\ldots+k_{i_t}$ is divisible by $p$, then the
passport $P$ will be called \emph{white-decomposable}. Analogously
can be defined a \emph{black-decomposable} passport. \pmn A
passport cannot be simultaneously white- and black-indecomposable.
This statement is a consequence of the following proposition.

\begin{prop} Let $\Pi$ is a partition of a number $x$, $l$ --- the
number of elements in $\Pi$ and $l>\frac x2$. Then for each $y$,
$0<y<x$, there exists a subset of $\Pi$ such, that the sum of its
elements is $y$.\end{prop}

\begin{proof} Let us assume that for some $y$ such subset doesn't
exist. For each subset $\Xi\in\Pi$ such, that the sum $S_\Xi$ of
its elements is $<y$, let $d_\Xi$ be the difference
$d_\Xi=y-S_\Xi$. And let for a subset $\Xi_0$ this difference be
minimal. \pmn The partition $\Pi$ has units as its elements
(otherwise, the sum of elements of $\Pi$ will be greater, than
$x$). All this units belong to the subset $\Xi_0$ (otherwise, the
move of one unit from $\Pi\setminus\Xi_0$ to $\Xi_0$ diminishes
$d_{\Xi_0}$). \pmn Let $a$ be the minimal element in
$\Pi\setminus\Xi_0$ and let $\Pi$ has $b$ units (and all of them
are in $\Xi_0$). We have,
$$a+b+2(l-b-1)=a-2-b+2l\leqslant x\Rightarrow a-2-b<0
\Rightarrow b>a-2.$$ Thus, $\Xi_0$ contains not less, than $a-1$
units, and by moving $a$ from $\Pi\setminus\Xi_0$ to $\Xi_0$ and
$a-1$ units from $\Xi_0$ to $\Pi\setminus\Xi_0$ we diminish
$d_{\Xi_0}$. Contradiction. \end{proof} \pmn Let $T$ be a plane
bipartite tree with the passport $P$ and let $b$ be a Shabat
polynomial for $T$ such, that white vertex $v_1$ be at the origin
and black vertex $u_1$ --- at the point $1$. This choice defines
positions of all other vertices. So, let $x_1=0,x_2,\ldots,x_n$ be
coordinates of white vertices $v_1,v_2,\ldots,v_n$, respectively,
and $y_1=1,y_2,\ldots,y_m$ be coordinates of black vertices
$u_1,\ldots,u_m$. Then
$$b(z)=\prod_{i=1}^n(z-x_i)^{k_i}.$$ As $c=b(1)$ is the value of
the Shabat polynomial $b$ in black vertices, then
$$b(z)-c=\prod_{i=1}^m(z-y_i)^{l_i}.$$ Let $K$ be the \emph{big
definition field} that contains coordinates of all white and black
vertices (and the number $c$ also). Let $\rho$ be a prime divisor
(see \cite{Bor}) of $K$ that divides prime $p$ and $v_\rho$ be the
corresponding valuation, i.e. $v_\rho(x)=a$, if $x=\rho^ay$, where
$y$ and $\rho$ are coprime. Let $O=\{x\in K\,|\,v_\rho(x)\geqslant
0\}$ be the set of $\rho$-integral numbers and $I=\{x\in
K\,|\,v_\rho(x)>0\}$ be the maximal ideal in $O$. Then $O/I$ is a
finite field of characteristic $p$. And let, at last,
$e_K=v_\rho(p)$ be the ramification index. \pmn In Section 2 we
prove the existence of \emph{normalized model}, i.e. such Shabat
polynomial of our tree $T$, that one white vertex is at the
origin, one black vertex is at $1$ and coordinates of all vertices
are $\rho$-integral. \pmn The main result of this work is Theorem
3.1: if the passport is white-indecomposable, then in the scope of
normalized model $v_\rho(x_i)=e_K/(n-1)$ for all $x_i\neq 0$.

\section{A normalized model}
\pn In this section we will construct a special Shabat polynomial
for our tree $T$ --- a \emph{normalized model} \cite{Zapponi}, and
will study its properties. \pmn If $\min_i v_\rho(y_i)=a<0$ and
$v_\rho(y_j)=a$, then we will perform a coordinate change: now
$x_i/y_j$, $i=1,\ldots,n$ are coordinates of white vertices and
$y_i/y_j$ are coordinates of black ones. Let us note that the
vertex $v_1$ remains at origin and at point 1 now is the vertex
$u_j$. We will continue to use notations $x_i$, $y_i$ and $c$ for
coordinates of white vertices, black vertices and the value of $b$
at 1. \pmn As all black coordinates now are $\rho$-integral, then
the polynomial
$$\prod_{i=1}^m(z-y_i)^{l_i}$$ has $\rho$-integral coefficients.
Thus, the polynomial
$$b(z)=\prod_{i=1}^m(z-y_i)^{l_i}+c$$ has $\rho$-integral
coefficients (because $x_1=0$). Hence, all white coordinates are
$\rho$-integral and the number $c$ is also $\rho$-integral.

\begin{defin} A Shabat polynomial $b$ of a tree $T$ is called its
\emph{normalized model} if
\begin{itemize} \item the leading coefficient of $b$ is $1$;
\item some prime number $p$ is fixed, which divides the number of
edges; \item in the big definition field some prime divisor $\rho$
is fixed, that divides $p$\,; \item some white vertex is in origin
and some black vertex is in the point $1$; \item coordinates of
all (black and white) vertices are $\rho$-integral.
\end{itemize} \end{defin} \pmn The existence of normalized model
was proved above. Let us consider in more details its arithmetic
properties. \pmn Let
$$b(z)=\prod_{i=1}^n(z-x_i)^{k_i}=z^N+a_{N-1}z^{N-1}+
\ldots+a_1z$$ be a normalized model of our tree $T$. On one hand
$$b'(z)=Nz^{N-1}+(N-1)a_{N-1}z^{N-2}+\ldots+a_1.$$
On the other hand
$$b'(z)=N\prod_{i=1}^n(z-x_i)^{k_i-1}\prod_{i=1}^m
(z-y_i)^{l_i-1}.$$ It means, that all coefficients $a_i$ (except,
maybe, $a_p,a_{2p},\ldots,a_{rp}$) belong to the ideal $I$. Thus,
$$b\text{ mod }I=z^{rp}+b_{r-1} z^{(r-1)p}+\ldots+b_1z^p,$$
where $b_i=a_{p\,i}\text{ mod }I\in O/I$. The polynomial
$t^r+b_{r-1}t^{r-1}+\ldots+b_1t$ has $r$ roots in $O/I$ and each
of them generates a root of the polynomial $b\text{ mod }\rho$ of
multiplicity $p$ (because $x\mapsto x^p$ is the Frobenius
automorphism in the field $O/I$). Thus, $N$ roots of the
polynomial $b$ are partitioned into $r$ subsets of cardinality $p$
each, and roots in each subset are congruent modulo $\rho$. \pmn
Let, for example, roots $x_1,\ldots,x_t$ are pairwise congruent
modulo $\rho$, but other roots $x_{t+1},\ldots,x_n$ are not
congruent to them. Then the number $k_1+\ldots+k_t$ is divisible
by $p$, i.e. the passport $P$ is white-decomposable. Thus, we have
the theorem.

\begin{theor} If there exists a pair of vertices $v_i$ and $v_j$
such, that their coordinates are not congruent modulo $\rho$, then
the passport $P$ is white-decomposable. If the passport is
white-indecomposable, then coordinates of all white vertices
belong to $I$. \end{theor} \pmn Analogous statement is valid for
black vertices.

\section{Arithmetic of trees with a white-indecomposable passport}
\pn Let $b$ be a normalized model of a tree $T$ with a
white-indecomposable passport $P$. As
$$b(z)\text{ mod }\rho=\prod_{i=1}^n(z-x_i)^{k_i}\text{ mod
}\rho=z^{N},$$ then $b(1)\equiv 1\text{ mod }\rho$. Hence,
coordinates of all black vertices are not in $I$. \pmn As
$$b'(z)=b(z)\sum_{i=1}^n\frac{k_i}{z-x_i}=N\prod_{i=1}^n
(z-x_i)^{k_i-1} \prod_{i=1}^m (z-y_i)^{l_i-1},\eqno(1)$$ then the
substitution $z=x_1=0$ gives us the relation
$$(-1)^{n-1}k_1\prod_{i=2}^n x_i=(-1)^{N-m}N\prod_{i=1}^m
y_i^{l_i-1}.$$ Thus,
$$\sum_{i=2}^n v_\rho(x_i)=e_K \Rightarrow (n-1)v\leqslant e_K,
\eqno(2)$$ where $v=\min_{i>1} v_\rho(x_i)>0$. Let $v_\rho(x_j)=v$
and let us consider new Shabat polynomial for our tree:
$$b_1(z)=x_j^{-N}b(x_jz)=\prod_{i=1}^n (z-\widetilde{x}_i)^{k_i}=
z^N+\widetilde{a}_{N-1}z^{N-1}+\ldots+\widetilde{a}_1z,$$ where
$\widetilde{x}_i=x_i/x_j$. On one hand,
$$b_1'(z)=x_j^{-N}x_j\,b'(x_jz)=N x_j^{1-n}\prod_{i=1}^n
(z-\widetilde{x}_i)^{k_i-1}\prod_{i=1}^m (x_j
z-y_i)^{l_i-1}.\eqno(3)$$ On the other hand,
$$b_1'(z)=Nz^{N-1}+(N-1)\widetilde{a}_{N-1}z^{N-2}+
\ldots+\widetilde{a}_1.$$ Some of the numbers $\widetilde{x}_i$ do
not belong to the ideal $I$. Thus, there exists a coefficient
$\widetilde{a}_t$ that also does not belong to $I$. As the
passport is white-indecomposable, then $t\text{ mod }p\neq 0$.
Hence, $v_\rho(t\,\widetilde{a}_t)=0$. But then from (3) we have
that $e_K\leqslant (n-1)v$. As $(n-1)v\leqslant e_K$ (relation
(2)), then we have the following theorem.

\begin{theor} If the passport is white-indecomposable, then
$$v_\rho(x_i)=e_K/(n-1),\quad \forall\,i>1.$$ \end{theor}

\pn The substitution $z=x_i$ in (1) and analogous reasoning give
us relation
$$v_\rho(x_i-x_j)=e_K/(n-1),\quad \forall\, 1\leqslant i<j\leqslant n.$$

\begin{rem} We want to obtain some estimation on the degree of the
field of definition (see \cite{Zv}) $L$ of a tree $T$. As $L$ is a
subfield of the field $K$, then let $\tau$ be a prime divisor in
$L$ that divides $p$ and is divisible by $\rho$. If $x,y\in L$,
then
$$\frac{v_\rho(x)}{v_\rho(y)}=\frac{v_\tau(x)}{v_\tau(y)}.$$ In
particular,
$$\frac{e_K}{v_\rho(x)}=\frac{e_L}{v_\tau(x)}, \eqno(4)$$ where
$e_L=v_\tau(p)$ is the ramification index of the field $L$.
\end{rem}

\begin{rem} If $N=p^sr$, $s>1$, then Theorem 2.1 is valid and
the statement of Theorem 3.1 is as follows:
$$v_\rho(x_i)=se_K/(n-1),\quad \forall\,i>1.$$
\end{rem}

\section{Examples}

\begin{ex} Let $T$ be a tree of diameter 4 with the central
black vertex of degree 4 and four white vertices of degrees
$a<b<c<d$ (here $N=a+b+c+d$). There are 6 trees with this
passport. We assume that there exists a prime $p$ such, that
$N\equiv 0\text{ mod }p$, $N\not\equiv 0\text{ mod }p^2$ and the
passport is white-indecomposable. \pmn The white vertex of degree
$d$ we put at origin and the central black vertex
--- at $1$. Then coordinates of all other vertices are uniquely
defined. \pmn Shabat polynomial for $T$ is a normalized model.
Indeed, otherwise we have to divide all coordinates by some number
with negative valuation. After that the coordinate of central
black vertex will be at $I$ and will have positive valuation. But
the passport is white-indecomposable, so coordinates of all black
vertices have zero valuation. We have a contradiction. \pmn Our
normalized model is defined over the field $L$ --- the definition
field of the tree $T$. Moreover, $x_c$ --- the coordinate of the
white vertex of degree $c$ belongs to $L$. Thus,
$$\frac{e_K}{v_\rho(x_c)}=\frac{e_L}{v_\tau(x_c)}=n-1=3.$$
It means that the degree of $L$ is not less, than $3$. But a
conjugate to any tree belongs to the same Galois orbit, hence, the
cardinality of each orbit is even. Hence, there is one orbit of
cardinality $6$.
\end{ex}

\begin{rem} The white-indecomposability can be obtained, if
$d>p\,(r-1)$. Thus, for example, all six trees with the passport
$$\langle 15,3,2,1\,|\,4,\underbrace{1,\ldots,1}_{17}\rangle$$
belong to one orbit. \end{rem}

\begin{rem} If $a=1,b=11,c=80,d=84$, then there are two Galois
orbits: one of cardinality $4$ and one of cardinality $2$. But
here the passport is white-decomposable, because
$1+11+80+84=11\cdot 16$. \end{rem}

\begin{rem} The above results hold if $N=p^sr$ and $s$ is coprime
with $3$ (see Remark 3.2). \end{rem}

\begin{rem} The same reasoning can be applied: a) in the case of the
passport $\langle a,b,c,c\,|\,4,1,\ldots,1\rangle$, where $a,b,c$
are pairwise different (there are 3 trees with such passport); b)
in the case of the passport $\langle a,b,c,c,c\,|\,5,1,\ldots,1
\rangle$, where $a,b,c$ are pairwise different (there are 4 trees
with such passport). \end{rem}

\begin{ex} Let us consider a tree $T$ with the passport
$$\langle p+1,2,\underbrace{1,\ldots,1}_{p-3}\,|\,
p-1,\underbrace{1,\ldots,1}_{p+1}\rangle,$$ where $p$ is a prime.
There are $p-2$ such trees. Let the white vertex of degree $p+1$
be at origin and the black vertex of degree $p-1$ be at point $1$.
The corresponding Shabat polynomial is a normalized model and its
definition field coincides with the definition field of $T$.
Number $a$ --- the coordinate of the white vertex of degree $2$
belongs to $L$. Thus,
$$\frac{e_K}{v_\rho(x)}=\frac{e_L}{v_\tau(x)}=p-2.$$ Hence, all trees
with this passport are in one Galois orbit.
\end{ex}
\pbn

\pbn

\end{document}